\documentclass{amsart}

\usepackage[utf8]{inputenc}
\usepackage[T1]{fontenc}
\usepackage{lmodern}
\usepackage[english]{babel}
\usepackage{tikz}
\usepackage{tikz-cd}
\usetikzlibrary{positioning,arrows}
\usepackage[colorlinks, linkcolor = blue, citecolor = blue, filecolor = black, urlcolor = blue]{hyperref}
\usepackage{bm}
\usepackage[english, capitalise]{cleveref}
\usepackage[vcentermath,enableskew]{youngtab}
\usepackage{enumitem}
\usepackage{combelow}
\usepackage{amssymb}

\theoremstyle{plain}
\newtheorem{prop}{Proposition}[section]
\newtheorem{theorem}{Theorem}
\newtheorem*{theorem*}{Theorem}
\newtheorem*{main}{Main Theorem}
\newtheorem*{lemma*}{Lemma}
\newtheorem{lemma}[prop]{Lemma}
\newtheorem{cor}[prop]{Corollary}
\newtheorem*{cor*}{Corollary}
\newtheorem{conj}[prop]{Conjecture}
\newtheorem*{conj*}{Conjecture}

\theoremstyle{definition}
\newtheorem{defi}[prop]{Definition}
\newtheorem{rem}[prop]{Remark}
\newtheorem*{rem*}{Remark}
\newtheorem{ex}[prop]{Example}
\newtheorem{cons}[prop]{Construction}

\numberwithin{equation}{section}

\newcommand{\mm}{\,\middle|\,}

\newcommand{\SL}{SL}
\newcommand{\SO}{SO}
\newcommand{\g}{\mathfrak{g}}

\renewcommand{\Re}{\mathbb{R}}
\newcommand{\Z}{\mathbb{Z}}
\newcommand{\N}{\mathbb{N}}
\newcommand{\C}{\mathbb{C}}

\newcommand{\Q}{\mathcal{Q}}

\newcommand{\PP}{\mathcal{P}}
\renewcommand{\L}{\mathcal{L}}

\newcommand{\W}{\mathcal{W}}

\DeclareMathOperator{\interior}{int}

\let\Im\relax
\DeclareMathOperator{\Im}{Im}
\DeclareMathOperator{\height}{ht}

\newcommand{\An}{\mathsf{A}_n}
\newcommand{\Bn}{\mathsf{B}_n}
\newcommand{\Cn}{\mathsf{C}_n}
\newcommand{\Dn}{\mathsf{D}_n}

\newcommand{\GG}{\mathsf{G}_2}

	\title{Reflexivity of Newton-Okounkov Bodies of Partial Flag Varieties}

	\author{Christian Steinert}
	
	\address{Mathematical Institute, Faculty of Mathematics and Natural Sciences, University of Cologne; Chair for Algebra and Representation Theory, RWTH Aachen University}
	
	\email{steinert@art.rwth-aachen.de}

\allowdisplaybreaks

\begin{document}

\begin{abstract}\noindent
	 Assume that the valuation semigroup $\Gamma(\lambda)$ of an arbitrary partial flag variety corresponding to the line bundle $\mathcal{L_\lambda}$ constructed via a full-rank valuation is finitely generated and saturated. We use Ehrhart theory to prove that the associated Newton-Okounkov body\,---\,which happens to be a rational, convex polytope\,---\,contains exactly one lattice point in its interior if and only if $\L_\lambda$ is the anticanonical line bundle. Furthermore we use this unique lattice point to construct the dual polytope of the Newton-Okounkov body and prove that this dual is a lattice polytope using a result by Hibi. This leads to an unexpected, necessary and sufficient condition for the Newton-Okounkov body to be reflexive.
\end{abstract}

\maketitle

\hypersetup{linkcolor=black,citecolor=black}


\section*{Introduction}

For quite some time researchers from different branches of mathematics have been interested in associating combinatorial objects (for example polytopes) to geometric objects (for example varieties). The textbook examples are of course toric varieties, where polytopes arise quite naturally encoding a lot of geometric information about the variety. Many people have been and are trying to make use of this fact by degenerating more complicated varieties into toric varieties\,---\,in particular Gonciulea and Lakshmibai \cite{GL}, Kogan and Miller \cite{KoM}, Caldero \cite{C}, Alexeev and Brion \cite{AB} as well as Feigin, Fourier and Littelmann \cite{FFL3}.

All of their approaches used polytopes that were already known to representation theorists because there has always been a strong interest in finding polytopes for representations to find new bases of these representations and thus the tools were already developed. Starting with the polytopes of Gelfand and Tsetlin in type $\An$ in \cite{GT} Berenstein and Zelevinsky defined {\em Gelfand-Tsetlin polytopes} for all classical Lie algebras in \cite{BZ1}. This approach lead to the construction of so called {\em string polytopes} for Lie algebras of arbitrary type that were studied by Littelmann in \cite{L} and Berenstein and Zelevinsky in \cite{BZ2}. A different {\em string polytope} has been defined by Nakashima and Zelevinsky in \cite{NZ}. Other prominent polytopes\,---\,usually called {\em Lusztig polytopes}\,---\,were defined by Lusztig in \cite{Lu}. A slightly different approach based on a conjecture by Vinberg led to the definition of {\em Feigin-Fourier-Littelmann-Vinberg polytopes} in types $\An$ \cite{FFL1} and $\Cn$ \cite{FFL2} by Feigin, Fourier and Littelmann. Gornitskii analogously defined {\em Gornitskii polytopes} in types $\Bn$ and $\Dn$ \cite{G2} as well as $\GG$ \cite{G}.

The most general approach to toric degenerations has been developed using Newton-Okounkov bodies, firstly defined by Okounkov in \cite{O} and \cite{O2}, by Lazarsfeld and Musta\cb{t}\u a \cite{LM}, Kaveh and Khovanskii \cite{KK} and Anderson \cite{A}. The formerly known representation theoretic polytopes can be realized as Newton-Okounkov bodies for some nice valuations, which has been shown by Kaveh \cite{Ka}, Kiritchenko \cite{Ki} and Fujita and Naito \cite{FN}. Most recently Kaveh and Manon analyzed the connection between Newton-Okounkov bodies and tropical geometry in \cite{KM}. A generalized method to construct most of the formerly mentioned polytopes in a representation theoretic setting\,---\,including Newton-Okounkov bodies\,---\,was developed by Fang, Fourier and Littelmann \cite{FFL4} via so called {\em birational sequences}.

In the context of this setting, reflexive polytopes appear naturally as Batyrev showed that they are in one-to-one correspondence to Gorenstein Fano toric varieties (see \cite[Proposition 2.2.23]{B}). Hence, finding reflexive polytopes means finding Gorenstein Fano toric degenerations.

Another viewpoint on polytopes associated to geometric objects arises in the theory of Mirror Symmetry. Most notably Batrev, Ciocan-Fontanie, Kim and van Straten used reflexive polytopes to construct mirror duality in \cite{BCKS} based on an idea of Batyrev \cite{B}. This approach was used by Rusinko \cite{R} to construct mirror duals for type $\An$ complete flag varieties using Littelmann's string polytopes, thereby recovering mirror families formerly described by Batyrev in \cite{B2}. The key point of his work was to prove that the duals of these polytopes are lattice polytopes in certain cases by observing that they contained a special lattice point in their interior. The goal of our paper is to understand this remarkable property in a more general setting.
Our main result is the following.

\begin{main}
	If the valuation semigroup $\Gamma(\lambda)$ associated to a partial flag variety $G/P$ via the $P$-regular dominant integral weight $\lambda$ and full-rank valuation $v$ is finitely generated and saturated, the following properties of the Newton-Okounkov body $\Delta(\lambda)$ are equivalent.
	\begin{enumerate}[label=(\roman*)]
		\item $\L_\lambda$ is the anticanonical line bundle over $G/P$.
		\item $\Delta(\lambda)$ contains exactly one lattice point $p_{\lambda}$ in its interior.
	\end{enumerate}
	Furthermore, in this case the polar dual\footnote{We always refer to the {\em polar dual} defined as $S^* := \left\lbrace y \in \Re^N \mm \langle x,y \rangle \leq 1 \text{ for all } x \in S\right\rbrace$ for an arbitrary set $S \subseteq \Re^N$. If $S$ is a polytope with the origin in its interior, then $S^*$ is a polytope.} of the translated Newton-Okounkov body $\Delta(\lambda) - p_{\lambda}$ is a lattice polytope.
\end{main}

Notice that this result applies to many of the formerly mentioned polytopes since most of them can be realized as Newton-Okounkov bodies of {\em nice} valuations.

Our statement bears resemblance to a result by Kaveh and Villella in \cite{KV} who were able to classify {\em anticanonical objects} in families of polyhedra associated to flag varieties purely via combinatorial conditions. However, their result needs stronger assumptions like Minkowski property of the occurring polytopes, which we do not need. 

This paper is structured as follows. For the definition of the objects in Newton-Okounkov theory and a collection of known facts see \cref{sec:no}. The proof of this theorem uses some results from Ehrhart theory and a result of Hibi \cite{Hi} that will be introduced in \cref{sec:ehr}. The proof itself is divided into multiple lemmata that will be stated and proved in \cref{sec:lem} and unified in the concluding proof of our Main Theorem in \cref{sec:proof}. An overview over certain applications and many examples of string polytopes finalize this paper in \cref{sec:more}. Most notably we will briefly elaborate on the following reflexivity criterion.

\begin{cor*}[\cref{cor:refl}]
	Under the assumptions of the Main Theorem the Newton-Okounkov body $\Delta(\lambda)$ is a reflexive polytope (after translation by a lattice vector) if and only if it is a lattice polytope and $\lambda$ is the weight of the anticanonical bundle over $G/P$.
\end{cor*}

Finally, we will give an example of a non-integral string polytope in type $\An$\,---\,hence disproving a conjecture by Alexeev in Brion\,---\,and pose a new conjecture on the integrality of string polytopes for very special reduced decompositions.

\textit{Acknowledgments.} This paper is a rewritten excerpt from my PhD thesis \cite{St} under the supervision of Peter Littelmann. I am very grateful to Peter Littelmann, Michel Brion, Xin Fang and Bea Schumann for many helpful discussions, for introducing me to the various mathematical concepts involved in this paper and for their invaluable, continued support.

\section{Facts from Newton-Okounkov Theory}\label{sec:no}

We will recall some important terminology regarding valuations and semigroups.

\begin{defi}
	Let $A$ be a $\C$-algebra and assume that $A$ is an integral domain. Fix a monoidal total ordering $\leq$ on $\Z^d$, i.e. a total ordering such that $a \leq b$ implies $a + c \leq b + c$ for all $a,b,c \in \Z^d$.
	
	A map $v \colon A \setminus \lbrace 0 \rbrace \to \Z^d$ is called a $\Z^d$-\textbf{valuation} on $A$ if it satisfies the following properties.
	\begin{enumerate}[label=(\roman*)]
		\item $v(\lambda f) = v(f)$ for all $\lambda \in \C^\times$ and $f \in A \setminus \lbrace 0 \rbrace$.
		\item $v(fg) = v(f)+v(g)$ for all $f,g \in A \setminus \lbrace 0 \rbrace$.
		\item $v(f+g) \geq \min\lbrace v(f), v(g) \rbrace$ for all $f,g \in A \setminus \lbrace 0 \rbrace$ such that $f+g \neq 0$.
	\end{enumerate}

	By slight abuse of notation we will denote the \textbf{valuation image} of $v$ as $\Im v := v(A\setminus \lbrace 0 \rbrace)$. We say that $v$ has \textbf{full rank} if the dimension of the $\Re$-linear span of the valuation image $\Im v \subseteq\Re^d$ equals the Krull dimension of $A$. We say that $v$ has \textbf{at most one-dimensional leaves} if the vector space
	\[ \left(\{ f \in A\setminus\{0\} \mid v(f) \geq s\}\cup \{0\}\right) / \left(\{ f\in A\setminus\{0\} \mid v(f) > s \} \cup \{0\}\right)\]
	is one-dimensional for every $s \in \Z^d$.
\end{defi}

\begin{rem}\label{rem:val}
	Condition (iii) yields the implication \[v(f+g) > \min \lbrace v(f), v(g) \rbrace \hspace{10pt}\Rightarrow\hspace{10pt} v(f) = v(g) \] for all $f,g \in A \setminus \lbrace 0 \rbrace$ such that $f + g \neq 0$. Indeed suppose that $v(f) \neq v(g)$ for some $f,g \in A \setminus \lbrace 0 \rbrace$. Without loss of generality let $v(f) < v(g)$. We can thus write $v(f) = v((f+g)+(-g)) \geq \min \lbrace v(f+g),v(g) \rbrace \geq \min \lbrace \min \lbrace v(f),v(g) \rbrace, v(g) \rbrace = v(f)$. Hence we have $v(f) = \min \lbrace v(f+g), v(g) \rbrace$. Since $v(f) < v(g)$ this implies $v(f+g) = v(f) = \min \lbrace v(f), v(g) \rbrace$.
\end{rem}

\begin{defi}
	Let $\Gamma$ be a semigroup in $\N \times \Z^d$. $\Gamma$ is called \textbf{finitely generated} if there exists a finite set of semigroup generators. We say that $\Gamma$ is \textbf{finitely generated in degree 1} if we can choose a finite set of semigroup generators with first coordinate equal to 1. $\Gamma$ is called \textbf{saturated} if for every $x \in \N \times \Z^d$ such that $mx \in \Gamma$ for some $m \in \Z_{>0}$ we find $x \in \Gamma$.
\end{defi}

We can now construct the main object of our studies.

\begin{cons}[Newton-Okounkov Body]
	Let $X$ be a normal projective variety of dimension $d$ and let $\L$ be an ample line bundle over $X$. Let \[R(X, \L) := \bigoplus_{m\in\N} H^0(X, \L^m)\] denote the associated ring of global sections and let $v\colon R(X, \L)\setminus\{0\} \to \Z^d$ be a valuation with respect to some monoidal order $\leq$ on $\Z^d$.
	We consider the graded monoid
	\[\Gamma(X,\L,v) := \bigcup_{m\in\N} \{ (m, v(f)) \mid f\in H^0(X, \L^m)\} \subseteq \N \times \Z^d,\]
	often called the \textbf{valuation monoid} or \textbf{valuation semigroup} with respect to $X$, $\L$ and $v$. It is indeed a monoid because for every $f\in H^0(X, \L^m)$ and $g\in H^0(X, \L^n)$ we have $fg\in H^0(X, \L^{m+n})$ and 
	\[ (m+n, v(fg)) = (m+n, v(f)+v(g)) = (m,v(f)) + (n, v(g)).\]
	Let 
	\[ \mathcal{C}(X, \L, v) := \overline{\operatorname{cone} \Gamma(X, \L, v)} \subseteq \Re_{\geq 0} \times \Re^d\]
	denote the closed cone over $\Gamma$. The \textbf{Newton-Okounkov body} $NO(X, \L, v)$ associated to $X$, $\L$ and $v$ is then defined as the intersection
	\[ \{1\} \times NO(X, \L, v) := \{ (1, x) \mid x \in \Re^d\} \cap \mathcal{C}(X, \L, v).\]
	
	In the context of flag varieties we will simply write $\Delta(\lambda)$ for the Newton-Okounkov body $NO(G/P, \L_\lambda, v)$ if the valuation has been fixed.
\end{cons}
	


We will now state some important results about valuations and Newton-Okounkov bodies. The first one is due to Kaveh and Manon \cite[Theorem 2.3]{KM}.

\begin{theorem}[Kaveh--Manon]
	Every full-rank valuation has at most one-dimensional leaves.
\end{theorem}

%

The next proposition is a conglomeration of results from Kaveh-Khovanskii \cite{KK} and Lazarsfeld-Musta\cb{t}\u a \cite{LM}. A proof can be found for example in \cite[Proposition 3.3.9]{St}.

\begin{prop}\label{prop:no}
	Let $X$ be a normal projective variety of dimension $d$, let $\L$ be an ample line bundle over $X$ and let $v\colon R(X, \L)\setminus\{0\} \to \Z^d$ be a full-rank valuation on the ring of global sections $R(X, \L)$ such that the associated semigroup $\Gamma(X, \L, v)$ is finitely generated and saturated. Then the Newton-Okounkov body $NO(X, \L, v)$ is a $d$-dimensional rational convex polytope with exactly $\dim H^0(X, \L, v)$ many lattice points and $n\cdot NO(X, \L, v) = NO(X, \L^n, v)$ for every $n\in\N$.
\end{prop}

\begin{rem}
	For the last claim we realize $R(X, \L^n)$ as a subring of $R(X, \L)$, so $v$ is a valuation on both rings.
\end{rem}

\section{Facts from Ehrhart Theory}\label{sec:ehr}

For the purpose of this section let $\PP$ denote a full-dimensional rational convex polytope in $\Re^d$. Let $\interior\PP = \PP\setminus\partial\PP$ denote its interior and $\PP^*$ its dual.

From the variety of interesting results that Ehrhart theory yields, we will only need the following two beautiful theorems. The first one is due to Ehrhart and Macdonald and can be found for example in \cite[Theorem 4.1]{BR}. It compares the number of lattice points of a polytope with the number of lattice points in its interior $\interior\PP$.

\begin{theorem}[Ehrhart-Macdonald Reciprocity]\label{thm:emr}
There exists a quasi-polynomial $L_\PP$ of degree $d$\,---\,called the \textbf{Ehrhart quasi-polynomial}\,---\,such that
\[L_{\PP}(n) = \#(n\PP\cap\Z^d) \hspace{10pt}\text{and}\hspace{10pt} (-1)^dL_\PP(-n) = \#(\interior n\PP \cap \Z^d).\]
\end{theorem}

Notice that the Ehrhart quasi-polynomial is not unique, but its evaluation on each integer is.

We will also use the notation $L_S(n) := \#(nS\cap\Z^d)$ for arbitrary subsets $S\subseteq\Re^d$, but this function might not share similar properties..

The second result due to Hibi gives a criterion on the integrality of the vertices of the dual polytope $\PP^*$. It can be found in \cite{Hi}.

\begin{theorem}[Hibi]\label{thm:hibi}
Suppose $0 \in \interior\PP$. Then $\PP^*$ is a lattice polytope if and only if \[\#(n\PP\cap\Z^d) = \#(\interior (n+1)\PP\cap\Z^d)\]
for every $n\in\N$.
\end{theorem}

\begin{rem}\label{rem:hibi}
	For computational purposes it is useful to notice that the condition in Hibi's \cref{thm:hibi} can be reformulated as
	\[ L_\PP(n) = (-1)^dL_\PP(-n-1) \hspace{10pt} \text{for all} n\in\N \]
	using Ehrhart-Macdonald Reciprocity (see \cref{thm:emr}).
\end{rem}

\section{Key Lemmata}\label{sec:lem}

Let us fix some notation.

Let $G$ be a simple algebraic group of rank $r$ with Lie algebra $\g$. Let $T$ be a maximal torus of $G$ and $B$ a Borel subgroup of $G$ containing $T$. Let $P$ be a parabolic subgroup of $G$ containing $B$ and let $L$ be the Levi subgroup of $P$ containing $T$. Let $\W := N_G(T)/T$ denote the Weyl group of $G$.

Let $\Phi$ be the set of roots of $G$ and let $\Phi^+$ be the subset of positive roots with respect to $B$. Denote the set of simple roots by $S = \lbrace \alpha_1, \ldots, \alpha_r\rbrace$. Let $N$ be the number of positive roots.

Let $\Lambda$ be the lattice of integral weights of $G$ and $\Lambda^+$ the subset of dominant integral weights with respect to $B$. Let $\omega_i \in \Lambda^+$ be the fundamental weight corresponding to $\alpha_i \in S$ and $\rho := \frac{1}{2} \sum_{\beta \in \Phi^+} \beta = \sum_{i=1}^{r}\omega_i$.

We know (see \cite[Theorem 8.4.3]{S}) that there exists a set of simple roots $I \subseteq S$ such that $P = \bigcup_{w \in \W_I} B\tilde{w}B$, where $\W_I \subseteq \W$ is the Weyl group generated by the simple reflections $\left\lbrace s_\alpha \mm \alpha \in I\right\rbrace$ and $\lbrace \tilde{w} \in N_G(T) \,|\, w \in \W\rbrace$ is a set of representatives for the Weyl group elements. Let $\langle I \rangle := \Phi \cap \lbrace \sum_{\alpha \in I} m_\alpha \alpha \,|\, m_\alpha \in \Z\rbrace$ and $\langle I \rangle^+ := \langle I \rangle \cap \Phi^+$. We define $\Lambda_P := \left\lbrace \lambda \in \Lambda \mm \langle \lambda, \alpha^\vee \rangle = 0 \text{ for all } \alpha \in I\right\rbrace$ and $\Lambda_P^+ := \Lambda_P \cap \Lambda^+$ as well as $\Phi^+_P := \Phi^+ \setminus \langle I \rangle^+$. Let $N_P$ be the cardinality of $\Phi^+_P$.

A dominant weight $\lambda \in \Lambda$ extends to a character of $P$ if and only if $\lambda \in \Lambda_P$. For every such $\lambda$ we define the one-dimensional vector space $\C_{-\lambda}$ with $P$-action given by $p.x := \lambda(p)^{-1}x$. We will consider the line bundle $\L_{P,\lambda} := G \times_P \C_{-\lambda}= (G \times \C_{-\lambda})/P$ over $G/P$ where the $P$-action on $G \times \C_{-\lambda}$ is given by $p.(g,x) := (gp,p^{-1}.x)$. We know that for a dominant weight $\lambda \in \Lambda_P^+$ the line bundle $\L_{P,\lambda}$ is ample if and only if $\lambda$ is $P$-regular, i.e. $\lambda \in \Lambda_P^+$ and $\langle \lambda, \alpha^\vee \rangle > 0$ for all $\alpha \in S \setminus I$. We will just write $\L_\lambda$ for $\L_{P,\lambda}$ if the parabolic is fixed. We will always implicitly exclude the trivial case $I = S$.

Our whole proof is based on the following computation of the coefficients of the Ehrhart series. Assumptions are as in the Main Theorem.

\begin{lemma}\label{lemma:ehrhart}
	$L_{\Delta(\lambda)}(n) = \prod_{\beta \in \Phi_P^+}\frac{\langle n\lambda + \rho, \beta^\vee\rangle}{\langle \rho, \beta^\vee \rangle}$ for all $n \in \Z$ and $\lambda \in \Lambda_P^+$.
\end{lemma}

Before we prove this statement, we want to recall a result by Kostant that is a combination of the famous Borel--Weil--Bott Theorem and Weyl's Character Formula. It can be found in \cite[Corollary 5.14]{K}.

\begin{theorem}[Kostant]\label{thm:kostant}
	Let $\lambda \in \Lambda_P^+$. Then $\dim H^0(G/P, \L_\lambda) = \prod_{\beta \in \Phi_P^+}\frac{\langle\lambda + \rho, \beta^\vee\rangle}{\langle \rho, \beta^\vee \rangle}$.
\end{theorem}

\begin{proof}[Proof of \cref{lemma:ehrhart}]
	By \cref{prop:no} (or equivalently \cite[Proposition 4.1]{LM}) we know that 
	\[n\cdot NO(G/P, \L_\lambda) = NO(G/P,\L_\lambda^n,v)\]
	and hence
	\[ L_{\Delta(\lambda)}(n) = \dim H^0(G/P,\L_\lambda^n) \]
	for every $n\in\N$. We want to show that this is equal to $\dim V(n\lambda)$. The claim then follows from Kostant's \cref{thm:kostant}.
	
	Consider the $n$-fold product map 
	\[ H^0(G/P, \L_\lambda) \times \ldots \times H^0(G/P, \L_\lambda) \to H^0(G/P, \L_\lambda^n).\]
	Notice that $\L_\lambda^n$ is ample, so $H^0(G/P, \L_\lambda^n)$ will be an irreducible $\g$-representation by Borel--Weil--Bott. Since this product map is $\g$-equivariant, its image must be a subrepresentation. The image is obviously not empty, so the product map is surjective. 
	
	Let $f_\lambda\in H^0(G/P, \L_\lambda)\simeq V(\lambda)^*$ be the global section corresponding to the lowest weight. Then the product $f_\lambda^n\in H^0(G/P, \L_{\lambda}^n)$ must be the lowest weight vector of $H^0(G/P, \L_{\lambda}^n)$. Since its weight is $-n\lambda$ we see that $H^0(G/P, \L_{\lambda}^n)$ is isomorphic to $V(n\lambda)^*$.
	
	Now the claim for positive integers follows from Kostant's version of Weyl's Dimension Formula in \cref{thm:kostant}. Notice that the right hand side of 
	\[
	L_{\Delta(\lambda)}(n) = \prod_{\beta \in \Phi_P^+}\frac{\langle n\lambda + \rho, \beta^\vee\rangle}{\langle \rho, \beta^\vee \rangle}\]
	can be seen as the evaluation of a polynomial at positive integers. Hence the Ehrhart quasi-polynomial must be a polynomial and since the two polynomials coincide on all positive integers, they coincide on all integers.
\end{proof}

\begin{ex}
	For the full flag variety and the anticanonical weight $2\rho$ we get
	\[L_{\Delta(\lambda)} = (2n+1)^N\]
	for every $n\in\Z$, where $N$ denotes the number of positive roots.
\end{ex}

The following lemmata state important results on the Weyl group $\W_I \subseteq \W$ corresponding to $P$. Let $w_I \in \W_I$ denote the longest word of $\W_I$.

\begin{lemma}\label{lemma:w_I}
	$w_I(\Phi_P^+) = \Phi_P^+$ and $w_I(\langle I \rangle^+) = - \langle I \rangle^+$.
\end{lemma}

\begin{proof}
	Since $\W_I$ is generated by all simple reflections $\lbrace s_\alpha \,|\, \alpha \in I\rbrace$ we know that $w_I(\langle I \rangle ) = \langle I \rangle$. Since $w_I \in \W$ we also have $w_I(\Phi) = \Phi$, thus $w_I(\Phi_P^+) \subseteq \Phi_P^+ \stackrel{\cdot}{\cup} -\Phi_P^+$. But for every $\beta = \sum_{\alpha \in S}m_\alpha \alpha \in \Phi^+_P$ there is at least one $\alpha \in S \setminus I$ such that $m_\alpha >0$. Since $w_I \in \langle s_\alpha \,|\, \alpha \in I \rangle$ this sign cannot be changed by $w_I$. This yields $w_I(\Phi^+_P) = \Phi^+_P$.
	
	The second part follows from the fact that $w_I$ is the longest word of the Weyl group $\W_I$ corresponding to the Levi $L_I$, so it sends positive roots of $L_I$ with respect to $B \cap L_I$ onto negative roots and vice versa.
\end{proof}

\begin{lemma}\label{lemma:acweight}
	The weight of the anticanonical bundle over $G/P$ is $\lambda_{G/P} = \rho + w_I(\rho)$.
\end{lemma}

\begin{proof}
	We know that the anticanonical bundle is the dual of the highest wedge power of the tangent space of $G/P$ whose weight is exactly $\sum_{\beta \in \Phi^+_P}\beta$. On the other hand we have
	\begin{align*}
		\rho + w_I(\rho) &= \frac{1}{2}\sum_{\beta \in \Phi^+}\beta + \frac{1}{2}\left(\sum_{\beta \in \langle I\rangle^+}w_I(\beta) + \sum_{\beta \in \Phi_P^+}w_I(\beta)\right)\\
		&= \frac{1}{2}\sum_{\beta \in \langle I\rangle^+}\beta + \frac{1}{2}\sum_{\beta \in \Phi_P^+}\beta - \frac{1}{2}\sum_{\beta \in \langle I\rangle^+}\beta + \frac{1}{2}\sum_{\beta \in \Phi_P^+}\beta 
		= \sum_{\beta \in \Phi_P^+}\beta		
	\end{align*}
	since $w_I$ permutes all elements of $\Phi_P^+$ and sends all the elements of $\langle I \rangle^+$ onto elements of $-\langle I \rangle^+$ bijectively as we proved in \cref{lemma:w_I}.
\end{proof}


The following lemma on root systems seems rather technical, but it is crucial to the proof of our Main Theorem.

\begin{lemma}\label{lemma:nonneg}
	Let $\lambda \in \Lambda_P^+$ be $P$-regular. Suppose there exists $\beta \in \Phi_P^+$ such that $\langle \lambda-\rho, \beta^\vee \rangle < 0$. Then there exists $\tilde\beta \in \Phi^+_P$ such that $\langle \lambda-\rho, \tilde{\beta}^\vee\rangle = 0$. 
\end{lemma}

To prove the lemma we need the following two lemmata.

\begin{lemma}\label{lemma:roots}
	Let $\beta = \sum_{i=1}^{r}m_i\alpha_i \in \Phi^+$ and $\height \beta > 1$. For every $i \in \lbrace 1 , \ldots , r \rbrace$ such that $m_i = 1$ there exists $j \neq i$ such that $\beta-\alpha_j \in \Phi$.
\end{lemma}

\begin{proof}
	We will prove the lemma by induction on $\height \beta$.
	
	For $\height \beta = 2$ we have nothing to prove since $\beta = \alpha_i + \alpha_j$ for some $i,j\in \lbrace 1 , \ldots, r \rbrace$.
	
	Now suppose that $\height\beta > 2$. Fix an $i \in \lbrace 1,\ldots,r\rbrace$ such that $m_i = 1$. If $\langle\beta,\alpha_i^\vee\rangle \leq 0$, we again have nothing to prove, because the proof of \cite[Lemma A of 10.2]{Hu} ensures that there exists at least one $j \in \lbrace 1, \ldots, r\rbrace$ such that $\langle\beta,\alpha_j^\vee\rangle >0$ which cannot be equal to $i$ by assumption. By \cite[Lemma 9.4]{Hu} this $j$ would then possess the desired property.
	
	So we only have to prove the case where $\langle\beta,\alpha_i^\vee\rangle > 0$. Because of \cite[Lemma 9.4]{Hu} this means that $\beta - \alpha_i$ is a (necessarily positive) root.
	
	Hence we know that the support of $\beta-\alpha_i$ is connected in the Dynkin diagram of $\g$. But because $m_i = 1$ we know that this support does not contain $\alpha_i$. This means that there exists only one simple root in the support of $\beta$ that is adjacent to $\alpha_i$, because otherwise the removal of $\alpha_i$ would result in a disconnected subgraph. Denote this adjacent simple root by $\alpha_j$. So for every $k \in \lbrace 1 , \ldots, r \rbrace \setminus \lbrace i,j\rbrace$ with $m_k > 0$ we have $\langle\alpha_k, \alpha_i^\vee\rangle = 0$. From $\langle\alpha_j,\alpha_i^\vee\rangle \leq -1$ and
	\[ 0 < \langle\beta,\alpha_i^\vee\rangle = m_i\langle\alpha_i,\alpha_i^\vee\rangle + m_j\langle\alpha_j,\alpha_i^\vee\rangle \leq 2 - m_j\]
	we conclude that $m_j < 2$ and thus $m_j = 1$. So we can use the induction hypothesis on $\beta-\alpha_i$ and get a $k \neq j$ such that $\beta-\alpha_i-\alpha_k$ is a root. Because $\beta-\alpha_i$ does not contain $\alpha_i$ in its support, we know that $k \neq i$. Thus we conclude
	\[ \langle \beta-\alpha_i-\alpha_k,\alpha_i^\vee \rangle = m_j \langle\alpha_j,\alpha_i^\vee\rangle - \langle \alpha_k, \alpha_i^\vee \rangle \leq -m_j - 0 = -1 < 0 \]
	and \cite[Lemma 9.4]{Hu} shows that $\beta - \alpha_k = \beta - \alpha_i - \alpha_k + \alpha_i$ is a (positive) root.
\end{proof}

\begin{lemma}\label{lemma:seq}
	Let $\beta \in \Phi_P^+$. There exists a sequence $(i_j)_{j \in \lbrace 1, \ldots, \height\beta\rbrace}$ in $\lbrace 1, \ldots, r\rbrace$ such that $\beta = \sum_{j = 1}^{\height \beta}\alpha_{i_j}$ and $\sum_{j = 1}^{k}\alpha_{i_j} \in \Phi_P^+$ for every $k \in \lbrace 1, \ldots, \height\beta\rbrace$.  
\end{lemma}

\begin{proof}
	We will prove the lemma by induction on $\height \beta$.
	
	If $\height \beta = 1$ there is nothing to prove.
	
	So let $h \in \N, h > 1,$ and suppose the lemma is true for every positive root $\beta' \in \Phi_P^+$ with $\height\beta' < h$. Let us now assume  $\beta \in \Phi_P^+$ with $\height\beta = h$. If no such $\beta$ exists we have nothing to prove.
	
	We know that there exists $\alpha \in S$ such that $\beta - \alpha \in \Phi^+$. If $\beta-\alpha \notin \Phi_P^+$ then $\beta$ must be of the form $\beta = \alpha + \sum_{\alpha' \in I} m_{\alpha'}\alpha'$. In this case \cref{lemma:roots} assures us that there exists another $\alpha' \in S$ such that $\beta - \alpha' \in \Phi$ and furthermore this root has to be in $\Phi_P^+$.
	
	So we can always find $\alpha \in S$ such that $\beta - \alpha \in \Phi_P^+$. By applying the induction hypothesis on that root we find the correct sequence $(i_j)_{j \in \lbrace 1, \ldots, h-1\rbrace}$ in $\lbrace 1, \ldots, r\rbrace$ for $\beta - \alpha$. Defining $i_h$ by $\alpha_{i_h} = \alpha$ will yield the desired sequence for $\beta$.
\end{proof}

We can now prove our last lemma and finish our preparations.

\begin{proof}[Proof of \cref{lemma:nonneg}]
	Let $\beta \in \Phi_P^+$ such that $\langle \lambda - \rho, \beta^\vee\rangle < 0$. Let $h := \height\beta$.
	
	Notice that $h > 1$ since for every simple root $\alpha \in \Phi_P^+$, i.e. $\alpha \in S\setminus I$, we have $\langle \lambda - \rho, \alpha^\vee\rangle \geq 0$ because $\lambda$ is $P$-regular.
	
	By \cref{lemma:seq} we find a sequence $(i_j)_{j \in \lbrace 1, \ldots, h\rbrace}$ in $\lbrace 1, \ldots, r\rbrace$ such that $\beta = \sum_{j = 1}^{h}\alpha_{i_j}$ and $\beta_k := \sum_{j = 1}^{k}\alpha_{i_j} \in \Phi_P^+$ for every $k \in \lbrace 1, \ldots, h\rbrace$. 
	
	Since $\langle \lambda - \rho, \alpha_{i_1}^\vee\rangle \geq 0$ there exists an index $k \in \lbrace 1, \ldots, h \rbrace$ such that $\langle \lambda - \rho, \beta_{k-1}^\vee\rangle \geq 0$ and $\langle \lambda - \rho, \beta_k^\vee\rangle < 0$. We have
	\begin{align*}
		0 \leq \langle \lambda - \rho, \beta_{k-1}^\vee\rangle = 2 \cdot \frac{\langle \lambda - \rho, \beta_k\rangle - \langle \lambda - \rho, \alpha_{i_k}\rangle}{\langle \beta_{k-1}, \beta_{k-1}\rangle} < - \frac{\langle \alpha_{i_k}, \alpha_{i_k} \rangle}{\langle \beta_{k-1}, \beta_{k-1} \rangle} \langle \lambda - \rho, \alpha_{i_k}^\vee \rangle.
	\end{align*}
	Since $\lambda$ is $P$-regular this is only possible if $\alpha_{i_k} \in I$, i.e. $\langle\lambda,\alpha_{i_k}^\vee\rangle = 0$, and thus  
	\[ 0 \leq \langle \lambda - \rho, \beta_{k-1}^\vee\rangle < \frac{\langle \alpha_{i_k}, \alpha_{i_k}\rangle}{\langle \beta_{k-1}, \beta_{k-1} \rangle}.\]
	This shows that there are only three possible values for $\langle\lambda-\rho, \beta_{k-1}^\vee\rangle$, since the fraction on the right side must be an element of $\lbrace \frac{1}{3}, \frac{1}{2}, 1, 2, 3\rbrace$.
	
	If $\langle\lambda-\rho, \beta_{k-1}^\vee\rangle = 0$ we have found the desired root $\tilde\beta = \beta_{k-1}$.
	
	If $\langle\lambda-\rho, \beta_{k-1}^\vee\rangle = 1$ we must have $\frac{\langle\alpha_{i_k}, \alpha_{i_k}\rangle}{\langle\beta_{k-1}, \beta_{k-1} \rangle} \in \lbrace 2,3\rbrace$. Set \[\tilde\beta := \alpha_{i_k} + \frac{\langle\alpha_{i_k}, \alpha_{i_k}\rangle}{\langle\beta_{k-1}, \beta_{k-1} \rangle} \beta_{k-1}\] as an element of the root lattice. We have 
	\begin{align*} \langle\lambda-\rho, \tilde{\beta}\rangle &= \langle\lambda-\rho, \alpha_{i_k}\rangle + \frac{\langle\alpha_{i_k}, \alpha_{i_k}\rangle}{\langle\beta_{k-1}, \beta_{k-1}\rangle} \langle\lambda-\rho, \beta_{k-1}\rangle\\
		&= - \frac{\langle\alpha_{i_k}, \alpha_{i_k}\rangle}{2}  + \frac{\langle\alpha_{i_k}, \alpha_{i_k}\rangle}{\langle\beta_{k-1}, \beta_{k-1}\rangle} \cdot \frac{\langle \beta_{k-1}, \beta_{k-1}\rangle}{2} = 0.
	\end{align*}
	We still have to show that $\tilde{\beta}$ is actually a root. By expanding $\langle \beta_k - \alpha_{i_k}, \beta_k-\alpha_{i_k} \rangle$ we find that
	\[\langle \beta_k, \alpha_{i_k}^\vee \rangle = 1 + \frac{\langle \beta_k, \beta_k\rangle}{\langle \alpha_{i_k}, \alpha_{i_k}\rangle} - \frac{\langle \beta_{k-1}, \beta_{k-1}\rangle}{\langle \alpha_{i_k}, \alpha_{i_k}\rangle}.\]
	Since the last summand is not an integer, we know that the second summand must not be an integer, too. But this means that $\beta_k$ and $\beta_{k-1}$ must have the same length because only two root lengths are allowed to occur in any irreducible root system (\cite[Lemma C of 10.4]{Hu}). We conclude that $\langle\beta_k, \alpha_{i_k}^\vee \rangle = 1$ and thus $\langle\beta_{k-1}, \alpha_{i_k}^\vee \rangle = -1$. This yields 
	\[ \langle\alpha_{i_k}, \beta_{k-1}^\vee\rangle = \frac{\langle\alpha_{i_k}, \alpha_{i_k}\rangle}{\langle\beta_{k-1}, \beta_{k-1}\rangle} \langle \beta_{k-1}, \alpha_{i_k}^\vee\rangle = -\frac{\langle\alpha_{i_k}, \alpha_{i_k}\rangle}{\langle\beta_{k-1}, \beta_{k-1}\rangle},  \]
	which implies that $\tilde{\beta}$ is a root using basic considerations on root strings (\cite[9.4]{Hu}). 
	
	The last possible case $\langle\lambda-\rho, \beta_{k-1}^\vee\rangle = 2$ can only occur if the root system is $\GG$, $\alpha_{i_k}$ is the long simple root and $\beta_{k-1}$ is a short positive root. Since their sum must again be a root, we know that $\beta_{k-1}$ has to be the short simple root. In that case we set $\tilde{\beta} = 2 \alpha_{i_k} + 3 \beta_{k-1} \in \Phi_P^+$ and calculate \begin{align*} \langle\lambda-\rho, \tilde{\beta}\rangle &= \frac{3}{2}\langle\beta_{k-1}, \beta_{k-1}\rangle \langle \lambda-\rho, \beta_{k-1}^\vee\rangle + \langle\alpha_{i_k}, \alpha_{i_k}\rangle\langle\lambda-\rho, \alpha_{i_k}^\vee\rangle\\ &= 3\langle\beta_{k-1},\beta_{k-1}\rangle - \langle\alpha_{i_k}, \alpha_{i_k}\rangle = 0,\end{align*}
	which concludes the proof.
\end{proof}

\section{Proof of the Main Theorem}\label{sec:proof}

We are now able to state the proof of our Main Theorem.

\begin{proof}[Proof of the Main Theorem]
	\cref{prop:no} tells us that $\Delta(\lambda) \subseteq \Re^{N_P}$ is a full-dimensional rational polytope and from \cref{lemma:ehrhart} we know that \[L_{\Delta(\lambda)}(n) = \prod_{\beta \in \Phi_P^+}\frac{\langle n\lambda + \rho, \beta^\vee\rangle}{\langle \rho, \beta^\vee \rangle}\]
	for all $n \in \Z$.
	
	Now suppose that $\Delta(\lambda)$ contains one unique lattice point in its interior. By Ehrhart-Macdonald Reciprocity in \cref{thm:emr} we have 
	\[1 = L_{\interior\Delta(\lambda)}(1) = (-1)^{N_P} L_{\Delta(\lambda)}(-1) = \prod_{\beta \in \Phi_P^+}\frac{\langle \lambda - \rho, \beta^\vee\rangle}{\langle \rho, \beta^\vee \rangle}. \]
	This implies that $\langle \lambda-\rho,\beta^\vee\rangle \neq 0$ for every $\beta \in \Phi_P^+$ and by \cref{lemma:nonneg} this actually means that $\langle \lambda-\rho,\beta^\vee\rangle > 0$ for all $\beta \in \Phi_P^+$. From \cref{lemma:w_I} we know that the longest word $w_I \in \W_I \subseteq \W$ permutes the elements of $\Phi_P^+$. Since it is a reflection, it leaves the scalar product invariant and by reshuffling factors we have
	\[1 =\prod_{\beta \in \Phi_P^+}\frac{\langle \lambda - \rho, \beta^\vee\rangle}{\langle \rho, \beta^\vee \rangle}=\prod_{\beta \in \Phi_P^+}\frac{\langle \lambda - \rho, (w_I\beta)^\vee\rangle}{\langle \rho, \beta^\vee \rangle} = \prod_{\beta \in \Phi_P^+}\frac{\langle w_I(\lambda - \rho), \beta^\vee\rangle}{\langle \rho, \beta^\vee \rangle}. \]
	Consider the integral weight $\mu =  \sum_{i=1}^{r} \mu_i\omega_i := w_I(\lambda-\rho)$. Every coefficient $\mu_i$ is strictly positive since $\langle \lambda-\rho,(w_I\beta)^\vee \rangle >0$ for every $\beta \in \Phi_P^+$\,---\,especially for every $\alpha \in S \setminus I$\,---\,and $\langle\lambda-\rho,(w_I\alpha)^\vee\rangle = -\langle\rho, (w_I\alpha)^\vee\rangle > 0$ for every $\alpha \in I$ because $w_I(\alpha) \in -\langle I \rangle^+$ by \cref{lemma:w_I}.
	
	This observation allows us to use the weighted inequality of arithmetic and geometric means to calculate
	\begin{align*}
	1 &= \prod_{\beta \in \Phi_P^+}\frac{\langle\mu, \beta^\vee\rangle}{\langle \rho, \beta^\vee \rangle} = \prod_{\beta \in \Phi_P^+}\frac{\sum_{i=1}^{r}\langle \omega_i, \beta^\vee\rangle \mu_i}{\langle \rho, \beta^\vee \rangle} \\
	& \geq \prod_{\beta \in \Phi_P^+} \left( \prod_{i=1}^{r} \mu_i^{\langle \omega_i,\beta^\vee\rangle} \right)^{\frac{1}{\langle \rho,\beta^\vee\rangle}} = \prod_{i=1}^{r}\left( \mu_i^{\sum_{\beta\in \Phi_P^+} \frac{\langle \omega_i,\beta^\vee\rangle}{\langle\rho,\beta^\vee\rangle}}\right).
	\end{align*}
	Since $\langle \omega_i,\beta^\vee\rangle \geq0$ for all $\beta \in \Phi_P^+$ with strict inequality at least once for every $i \in \lbrace 1, \ldots, r\rbrace$, we have strictly positive coefficients $a_1, \ldots, a_r \in \Re_{>0}$ such that 
	\[1 \geq \mu_1^{a_1} \cdots \mu_r^{a_r}.\]
	Since all of the $\mu_i$ are strictly positive integers, this inequality can only hold if $\mu_i = 1$ for all $i \in \lbrace1,\ldots,r\rbrace$ and then it is in fact an equality. But this means that $w_I(\lambda-\rho) = \mu = \sum_{i=1}^{r}\omega_i = \rho$ and thus $\lambda = \rho + w_I(\rho)$. By \cref{lemma:acweight} this is the weight of the anticanonical line bundle over $G/P$, which proves the first direction.
	
	In fact, we also proved the other direction on the way because we noticed that $\mu = \rho$ if $\lambda$ is the weight of the anticanonical bundle, which yields $L_{\interior \Delta(\lambda)}(1) = \prod_{\beta \in \Phi_P^+}\frac{\langle\mu,\beta^\vee\rangle}{\langle\rho,\beta^\vee\rangle} = 1 $ if we apply the above calculations in opposite order.
	
	So what is left to prove is the final implication of the theorem. Let $\lambda = \rho + w_I(\rho)$ be the weight of the anticanonical line bundle over $G/P$. We calculate
	\begin{align*}
		(-1)^{N_P} L_{\Delta(\lambda)}(-n-1) &= \prod_{\beta \in \Phi_P^+}\frac{\langle (n+1)\lambda - \rho, \beta^\vee\rangle}{\langle \rho, \beta^\vee \rangle}\\
		&= \prod_{\beta \in \Phi_P^+}\frac{\langle n\rho + \rho + nw_I(\rho) + w_I(\rho) - \rho, \beta^\vee\rangle}{\langle \rho, \beta^\vee \rangle}\\
		&= \prod_{\beta \in \Phi_P^+}\frac{\langle n(\rho + w_I(\rho)) + w_I(\rho), (w_I\beta)^\vee\rangle}{\langle \rho, \beta^\vee \rangle}\\
		&= \prod_{\beta \in \Phi_P^+}\frac{\langle n(w_I(\rho) + \rho) + \rho, \beta^\vee\rangle}{\langle \rho, \beta^\vee \rangle}\\
		&= \prod_{\beta \in \Phi_P^+}\frac{\langle n\lambda + \rho, \beta^\vee\rangle}{\langle \rho, \beta^\vee \rangle} = L_{\Delta(\lambda)}(n)
	\end{align*}
	for all $n \in \N$. It is clear that the Ehrhart polynomial of a polytope is invariant under translation of the polytope via a lattice vector. Hence Hibi's \cref{thm:hibi} concludes the proof.
\end{proof}

\section{Applications}\label{sec:more}

We have the following two immediate corollaries to our Main Theorem under the assumptions of the Main Theorem.

\begin{cor}\label{cor:refl}
	The Newton-Okounkov body $\Delta(\lambda)$ is a reflexive polytope (after translation by a lattice vector) if and only if it is a lattice polytope and $\lambda$ is the weight of the anticanonical bundle over $G/P$.
\end{cor}

Let $GT(\lambda)$ denote the {\em Gelfand-Tsetlin polytope} as defined for type $\An$ in \cite{GT} and for type $\Cn$ in \cite{BZ1}. Let $FFLV(\lambda)$ denote the {\em Feigin-Fourier-Littelmann-Vinberg polytope} as defined for type $\An$ in \cite{FFL1} and for type $\Cn$ in \cite{FFL2}. Let $G(\lambda)$ denote the {\em Gornitskii polytope} as defined for type $\GG$ in \cite{G}.

\begin{cor}
	Let $G$ be of type $\An$ or $\Cn$, let $G/P$ be a flag variety and let $\lambda \in \Lambda^+_P$. Then $GT(\lambda)$ and $FFLV(\lambda)$ are reflexive (after translation by a lattice vector) if and only if $\lambda$ is the weight of the anticanonical bundle over $G/P$.
	
	Let $G$ be of type $\GG$, let $G/P$ be an arbitrary partial flag variety and let $\lambda \in \Lambda^+_P$. Then $G(\lambda)$ is reflexive (after translation by a lattice vector) if and only if $\lambda$ is the weight of the anticanonical bundle over $G/P$.
\end{cor}

Finally we want to study one of the biggest classes of examples\,---\,namely the {\em string polytopes} $\Q_{\underline{w_0}}(\lambda)$ as defined in \cite{L} using notation from \cite{AB}. 
As a special case we have the following observation for the full flag variety in type $\An$ that has already been proved by Rusinko directly in \cite[Theorem 7]{R}.

\begin{cor}[Rusinko]
	Let $G=\SL_{n+1}$. Then the dual of the string polytope $\Q_{\underline{w_0}}(2\rho)$ is a lattice polytope (after translation) for every reduced decomposition $\underline{w_0}$.
\end{cor}

Of course one would like to give a precise criterion when the string polytope of a partial flag variety is reflexive. But this is not solvable at the moment because it is not known when the string polytope is a lattice polytope, even for nice reduced decompositions and minuscule weights. We want to conclude our paper by illustrating this problem in the following three examples and stating a conjecture that would partially solve this problem.

All calculations in this section have been achieved using SageMath \cite{sage}. More detailed examples are available in my PhD thesis \cite[Section 6.3]{St}.

Our first example will answer a prominent question regarding string polytopes in type $\An$ by giving a counter-example to the following conjecture as formulated by Alexeev and Brion in \cite[Conjecture 5.8]{AB}.

\begin{conj}[Alexeev, Brion]
	For $G=\SL_{n+1}$ and any reduced decomposition $\underline{w_0}$, the string polytope $\Q_{\underline{w_0}}(\lambda)$ is a lattice polytope for every $\lambda \in \Lambda^+$.
\end{conj}

This conjecture has been verified by Alexeev and Brion for all $n \leq 4$ in \cite{AB}. We will see that it does not hold anymore for $n = 5$.
	
\begin{ex}\label{ex:36}
	Let $G = \SL_6$ and consider the Grassmannian $G/P = \mathrm{Gr}(3,6)$. Choose the reduced decomposition $\underline{w_0} = s_1s_3s_2s_1s_3s_2s_4s_3s_2s_1s_5s_4s_3s_2s_1$. Notice that this reduced decomposition arises from the standard reduced decomposition of \cite{L} by applying two 3-moves (and two 2-moves). Hence we have multiple ways of calculating the string polytopes in addition to the construction by Berenstein and Zelevinsky in \cite[Theorem 3.14]{BZ2}. We find that there exists one non-integral vertex of $\Q_{\underline{w_0}}(\omega_3)$. Luckily this non-integral vertex has half-integral coordinates, so the string polytope for the weight of the anticanonical bundle $\lambda_{\mathrm{Gr}(3,6)} = 6\omega_3$ is again a lattice polytope. 
	
	But this magic trick does not happen every time, since we can enlarge this example in $\mathsf{A}_5$ to a whole class of examples for arbitrary $n$ by using the reduced decomposition $\underline{w_0} = (s_1s_3s_2s_1s_3s_2)(s_4s_3s_2s_1)(s_5s_4s_3s_2s_1) \cdots (s_ns_{n-1} \cdots s_2s_1).$
	The respective string polytope $\Q_{\underline{w_0}}(\omega_3)$ will not be a lattice polytope for $n \geq 5$. In particular for $n = 6$ we can calculate that $\Q_{\underline{w_0}}(\omega_3)$ has half-integral vertices. Thus even for the weight of the anticanonical bundle $\lambda_{\mathrm{Gr(3,7)}} = 7\omega_3$ over $Gr(3,7)$ the string polytope $\Q_{\underline{w_0}}(7\omega_3) = 7\cdot \Q_{\underline{w_0}}(\omega_3)$ will not be a lattice polytope.
\end{ex}

\begin{rem}
	It seems that this observation is connected to the fact that the string polytopes for the reduced decomposition $\underline{w_0} = s_1s_3s_2s_1s_3s_2$ in $\mathsf{A}_3$ do not fulfill the {\em Minkowski property} (also called {\em Integral Decomposition Property}), i.e. for arbitrary $\lambda, \mu \in \Lambda^+$ the lattice points in the string polytope $\Q_{\underline{w_0}}(\lambda+\mu)$ cannot be written as sums of lattice points from $\Q_{\underline{w_0}}(\lambda)$ and $\Q_{\underline{w_0}}(\mu)$. This implies that there exists $\lambda \in \Lambda^+$ such that $\Q_{\underline{w_0}}(\lambda)$ contains lattice points that are not sums of lattice points of the fundamental string polytopes. And although $\mathsf{A}_3$ and $\mathsf{A}_4$ are too small to create non-integral string polytopes, this already foreshadows that something interesting might happen for higher rank.
\end{rem}

\begin{rem}
	In \cite{RW} Rietsch and Williams constructed Newton-Okounkov bodies for Grassmannians using plabic graphs. In some cases their construction leads to non-integral polytopes\,---\,the first two appearing for the same Grassmannian $\mathrm{Gr}(3,6)$. Both of these polytope have a single non-integral vertex as well. I want to thank Valentin Rappel for pointing out this remarkable similarity.

	It is natural to ask whether the string polytope from \cref{ex:36} and the respective Rietsch-Williams polytopes are actually unimodularly equivalent. Joint with Lara Boßinger we were able to show that this is in fact true for one of the two Rietsch Williams polytopes but not true for the other one. It would be very interesting to understand the reason behind this sporadic equivalence.
\end{rem}

So we have seen that in type $\An$ only non-standard reduced decomposition can\,---\,and indeed will\,---\,give rise to non-integral string polytopes. In other types the situation is even more challenging since the standard reduced decompositions of \cite{L} will already provide those as we will see in the next example.

\begin{ex}\label{ex:b2}
	Let $G = \SO_5$ and choose $\underline{w_0}$ to be the standard reduced decomposition from \cite[Section 6]{L}, which is $\underline{w_0} = s_2 s_1 s_2 s_1$, where $\alpha_2$ denotes the short root. Let $\lambda = \omega_2$. The corresponding string polytope contains one half-integral vertex. Even more, the affine hull of its lattice points is two-dimensional while the polytope itself is three-dimensional.
	
	Since the vertices have at worst half-integral coordinates, we see that the string polytope for the weight of the anticanonical bundle $\lambda_{G/P(\alpha_1)} = 4\omega_2$ over $G/P(\alpha_1)$ will be a lattice polytope and by our theorem reflexive after translation by the lattice vector $(1,2,3,0)^T$. 
	
\end{ex}

\begin{rem}
	\cref{ex:b2} contradicts \cite[Theorem 4.5]{AB}, which claims that the string polytope for any (co)minuscule weight and any reduced decomposition must be a lattice polytope. Peter Littelmann and Michel Brion were able to solve this contradiction by finding a fault in the proof of said claim. Essentially the problem arises by applying a result of Caldero and Littelmann on standard monomials. In the proof of \cite[Theorem 4.5]{AB}, the authors construct a sequence of subwords of the longest word of the Weyl group of the form
	\[   \underline{w_0} = s_{i_1} \cdots s_{i_N} \geq s_{i_{j_1}} \cdots s_{i_N} \geq \ldots \geq s_{i_{j_n}} \cdots s_{i_N}    \]
	but the result of Caldero and Littelmann would actually require a sequence of the form 
	\[   \underline{w_0} = s_{i_1} \cdots s_{i_N} \geq s_{i_1} \cdots s_{i_{k_1}} \geq \ldots \geq s_{i_1} \cdots s_{i_{k_n}}.   \]
	I want to thank Peter Littelmann and Michel Brion for explaining this problem.
\end{rem}

From the previous two examples we can already see that sticking to the standard reduced decompositions of \cite{L} might yield some useful results. Known results and many calculations for string polytopes in classical types suggest the following.

\begin{conj}\label{conj:lattice}
	Let $G$ be a complex classical group, let $\lambda \in \Lambda^+$ and let $\underline{w_0}^\mathrm{std}$ be the standard reduced decomposition of the longest word of the Weyl group of $G$ as stated in \cite{L}. Then $\Q_{\underline{w_0}^\mathrm{std}}(\lambda)$ is a lattice polytope if and only if one of the following conditions hold.
	\begin{enumerate}[label=(\roman*)]
		\item $G$ is of type $\An$,
		\item $G$ is of type $\Bn$ and $\langle \lambda, \alpha_n^\vee\rangle \in 2\Z$,
		\item $G$ is of type $\Cn$ or
		\item $G$ is of type $\Dn$ and $\langle \lambda, \alpha_{n-1}^\vee\rangle + \langle \lambda, \alpha_n^\vee \rangle \in 2\Z$ or $n<4$.
	\end{enumerate}
\end{conj}

Together with \cref{cor:refl} this conjecture would imply the following.

\begin{conj}\label{conj:reflexive}
	Let $G$ be a complex classical group, let $G/P$ be a partial flag variety and let $\underline{w_0}^\mathrm{std}$ be the standard reduced decomposition of the longest word of the Weyl group of $G$ as stated in \cite{L}. Let $\lambda \in \Lambda_P^+$. Then $\Q_{\underline{w_0}^\mathrm{std}}(\lambda)$ is reflexive (after translation by a lattice vector) if and only if $\lambda$ is the weight of the anticanonical bundle over $G/P$.
\end{conj}

\begin{rem}
	The implication is due to the fact that the conditions in \cref{conj:lattice} are fulfilled, whenever the irreducible highest weight $\mathfrak{g}$-representation $V(\lambda)$ integrates to a representation of the underlying simple algebraic group $G$.
\end{rem}

\begin{rem}
	These conjectures are actually true and their proofs can be found in my PhD thesis \cite{St}. A comprehensive version of this material will appear in one (or two) forthcoming papers.
\end{rem}

In the exceptional cases the situation is even more unclear, as can be seen in our final example.

\begin{ex}
	Let $G$ be of type $\GG$. Consider the anticanonical bundle over the full flag variety $G/B$. We choose $\underline{w_0} = \underline{w_0}^\mathrm{std} = s_1 s_2 s_1 s_2 s_1 s_2$ starting with the short root. 
	Following \cite[Section 2]{L} one calculates that the vertices of $\Q_{\underline{w_0}}(2\rho)$ lie in $\frac{1}{3}\Z$. Hence $\Q_{\underline{w_0}}(2\rho)$ is not a lattice polytope and thus not reflexive even after translation by the unique interior lattice point $(1,2,5,3,4,1)^T$. 
	
	In fact, one can show that for all but one combination of parabolics and reduced decompositions, the respective anticanonical string polytope will not be a lattice polytope. The only exception is the lattice polytope $\Q_{s_2 s_1 s_2 s_1 s_2 s_1}(2\rho)$.
\end{ex}

\end{document}